\documentclass[12pt]{amsart}
\usepackage{fullpage}
\usepackage[title]{appendix}
\usepackage{graphicx}
\usepackage{subcaption}
\usepackage{bm} 
\usepackage{amsthm}
\usepackage{mathrsfs} %%pour le style calligraphie en mathscr
\usepackage{amsfonts}
\usepackage{amssymb}
\usepackage{diagbox}
\usepackage{mathtools}
\usepackage{longtable} %tables with pagebreak
\usepackage{MnSymbol}
\usepackage{indentfirst}
\usepackage{listings,xcolor}
\usepackage[all]{xy}
\usepackage[colorlinks=true,linkcolor=blue]{hyperref}
\usepackage{ytableau}
\usepackage{tikz-cd}
\usepackage{enumerate}
\theoremstyle{plain}
\newtheorem{theorem}{Theorem}[section]
\newtheorem{thm}{Theorem}

\newtheorem{lemma}[theorem]{Lemma}
\newtheorem{corollary}{Corollary}[section]
\newtheorem{proposition}[theorem]{Proposition}

\theoremstyle{remark}

\newtheorem{remark}{Remark}[section]

\makeatletter

\newcommand{\Rmnum}[1]{\expandafter\@slowromancap\romannumeral #1@}
\makeatother

\def\ri{{\rm i}}

\def\li{{\rm Li}}

\def\bq{\mathbb Q}
\def\br{\mathbb R}
\def\bn{\mathbb N}
\def\bz{\mathbb Z}
\def\bc{{\mathbb C}}

\newmuskip\pFqmuskip

\numberwithin{equation}{section}
\allowdisplaybreaks[2] %allow page breaks in align enviroment

\begin{document}

\title[Modular transformations of bilateral $q$-series]{Theta functions and transformations of bilateral basic hypergeometric series}
\author{Nian Hong Zhou}

\address{N. H. Zhou: School of Mathematics and Statistics, The Center for Applied Mathematics of Guangxi, Guangxi Normal University, Guilin 541004, Guangxi, PR China}
\email{nianhongzhou@outlook.com; nianhongzhou@gxnu.edu.cn}%

\thanks{This paper was partially supported by the National Natural Science Foundation of China (No. 12301423).}%

\begin{abstract}
We establish new transformation formulas involving theta functions and certain bilateral basic hypergeometric series. From these formulas, we construct companion $q$-series for a class of $q$-series such that the asymptotic expansion of their quotient admits a simple closed form. This allows us to prove several conjectures of McIntosh on asymptotic transformations of $q$-series. Moreover, our results extend some identities of Ramanujan and McIntosh.  
\end{abstract}
\maketitle

\section{Introduction}\label{sec1}
\subsection{Background}
Throughout this paper, we assume that $q\in\bc$ and $0<|q|<1$.
For any indeterminate $a$ and any real number $c$, the
\textit{$q$-shifted factorial} is defined by
\[
(a)_\infty:=(a;q)_\infty:=\prod_{j\ge 0}(1-aq^j),
\quad\text{and}\quad
(a)_c:=\frac{(a)_\infty}{(aq^c)_\infty}.
\]
We use the compact notation
\[
(a_1,\ldots,a_m)_c:=\prod_{j=1}^m(a_j)_c
\]
for products of $q$-shifted factorials, where $m$ is a positive integer
and $c\in\br\cup\{\infty\}$.

%The well--known Euler's identity (see \cite{euler2007various} and \cite[p.19]{MR557013} for example ) state that
%\begin{align}\label{eqm0}
%1+\sum_{n\ge 1}\frac{z^n q^{\frac{n(n-1)}{2}}}{(1-q)(1-q^2)\cdots(1-q^n)}=\prod_{n\ge 0}(1+zq^n).
%\end{align}
%Letting $z=q$ in \eqref{eqm0} we obtain a strict partition generating function identity
%\begin{align}\label{eqm1}
%\prod_{n\ge 1}(1+q^n)=1+\sum_{n\ge 1}\frac{q^{\frac{n(n+1)}{2}}}{(1-q)(1-q^2)\cdots(1-q^n)}.
%\end{align}
\medskip

One of Euler's well-known identities (cf.\ \cite[p.19, Eq. (2.2.6)]{MR557013}) states that
\begin{align}\label{eqm1}
(-z)_\infty=\sum_{n\ge 0}\frac{z^n q^{\binom{n}{2}}}{(q)_n},
\end{align}
for any $z\in\bc$. Following Ramanujan's last letter to Hardy
(cf.\ Watson \cite{MR1573993}), we call sums such as the one in
\eqref{eqm1} \emph{Eulerian forms}. In the same letter, Ramanujan listed
$17$ examples of functions in Eulerian form, which he called
\emph{mock theta functions}. The first three pages, in which Ramanujan
explained what he meant by a ``mock theta function,'' are very obscure.
Hardy comments that a mock theta function is a function defined by a
$q$-series convergent when $|q|<1$, for which one can calculate asymptotic
formulae, as $q$ tends to a ``rational point'' $e^{2\pi i s/r}$ of the unit
circle, with the same degree of precision as those furnished for ordinary
theta functions by the theory of linear transformations
(cf.\ Watson \cite[p.57]{MR1573993}). Ramanujan also noted that, for other
$q$-series in Eulerian form, approximations analogous to those for mock
theta functions may not exist. For instance, he stated that if
$q=e^{-t}$ with $t\to 0^+$, that is, if $q$ tends to $1$, then for all
integers $m\ge 2$ we have
\[
\sum_{n\ge 0}\frac{q^{\binom{n+1}{2}}}{(q)_n^2}
=
\sqrt{\frac{t}{2\pi\sqrt{5}}}
\exp\left(
\frac{\pi^2}{5t}
+
\frac{t}{8\sqrt{5}}
+
\sum_{k=2}^m c_k t^k
+
O(t^{m+1})
\right),
\]
where infinitely many coefficients $c_k$ are nonzero. Ramanujan said that
in this case ``the exponential factor does not close,'' but a rigorous proof
has not yet been found.

\medskip
In \cite[p. 60]{MR1573993}, Watson proved that
\[
\sum_{n\ge 0}\frac{q^{\binom{n+1}{2}}}{(q)_n^2}
=
\frac{1}{(q)_\infty}
\sum_{n\ge 0}\frac{q^{2n^2+n}}{(q^2;q^2)_n}.
\]
By Euler's pentagonal number theorem, $(q)_\infty$ can be expressed as a
theta series, and it satisfies the following transformation formula
(cf.\ \cite[Equation (2)]{MR1310726}):
\begin{equation}\label{eqdedk}  
(e^{-t}; e^{-t})_\infty
=
(2\pi/t)^{1/2}
e^{\frac{t}{24}-\frac{\pi^2}{6t}}
\left(e^{-4\pi^2/t}; e^{-4\pi^2/t}\right)_\infty, 
\end{equation}  
where ${\rm Re}(t)>0$. Thus, to prove Ramanujan's claim, it remains to study the asymptotic behavior of the series on the right-hand side of Watson's identity. Let the \emph{dilogarithm function}
$\operatorname{Li}_2(z)$ be defined by
\[
\operatorname{Li}_2(z)=\sum_{n\ge 1}\frac{z^n}{n^2},
\]
for $|z|\le 1$. McIntosh \cite{MR1310726}
(see also \cite[Theorem 1]{MR1618298}) established the following theorem.

\begin{thm}\label{macmain}
Let $a,b>0$, let $c\in\br$, and let $q=e^{-t}$ with $t\to 0^+$. Then
\begin{equation*}
\sum_{n\ge 0}\frac{a^n q^{bn^2+cn}}{(q)_n}
=
\frac{z^c}{\sqrt{z+2b(1-z)}}
\exp\left(
\frac{\operatorname{Li}_2(1-z)+b\log^2 z}{t}
+
\sum_{1\le k\le m}c_{a,b,c}(k)t^k
+
O(t^{m+1})
\right),
\end{equation*}
for all $m\in\bn$, where $z$ denotes the positive root of
$az^{2b}+z-1=0$, and the constants $c_{a,b,c}(k)$ depend only on
$k$, $a$, $b$, and $c$.
\end{thm}
It should be pointed out that, for general $k$, the constants
$c_{a,b,c}(k)$ cannot be determined explicitly, as shown in
\cite{MR1310726}. McIntosh \cite[Section 7]{MR1310726} further showed that
the method used in the proof of Theorem \ref{macmain} can be applied to a
wide variety of unimodal series under certain limiting conditions. For
further discussion, see \cite[Theorem 2]{MR1618298}.

\medskip
%The following conjecture, initially posited by McIntosh \cite[p.415]{MR1310726}, will be used in our subsequent analysis. For clarity and self-containedness, we formally restate this conjecture as follows:
%\begin{conjecture}
%For $0 < b < 1/2$ and $c \in \mathbb{R}$, the asymptotic expansion \eqref{eqmmm} holds for all $a \in \mathbb{R}$. \end{conjecture}
%Suppose a $q$-series $\ch(q)$ has the asymptotic behavior
%$$\ch(q) = \exp\left(w_\ch \log t+\sum_{-1 \leq k \leq m} c_\ch(k) t^k + O(t^{m+1})\right), \quad \text{as } t \to 0^+,$$
%for all $m \in \bn$, where the coefficients $w_\ch\in\bq$ and $c_{\ch}(k) \in \bc$ are constants depending on $\ch(\cdot)$. Following McIntosh \cite{MR1310726, MR1618298}, we say the asymptotic expansion of $\ch(q)$ is \emph{closed} if $c_\ch(k) \neq 0$ holds only for finitely many $k$.
As an application of asymptotic expansions of $q$-series, McIntosh
\cite{MR1310726, MR1618298} constructed \emph{companion $q$-series} for
certain $q$-series, such that the quotient of each series and its companion
has an asymptotic expansion of a simple closed form. For simplicity, we
define
\begin{align}\label{eqmm1}
f_1(a,b,c;q):=\sum_{n\ge 0}\frac{a^nq^{bn^2+cn}}{(q)_n},
\end{align}
and
\begin{align*}
\tilde{f}_1(a,b,c;q)
:=
f_1\left(
-a^{-\frac{1}{2b}},
\frac12-\frac{1}{4b},
\frac12-\frac{c}{2b};
q
\right),
\end{align*}
where $a>0$, $b>1/2$, and $c\in\br$. McIntosh
\cite[p. 416]{MR1618298} stated that numerical computations support the
conjecture
\begin{align}\label{eqmm10}
\frac{f_1(a,b,c;e^{-t})}{\tilde{f}_1(a,b,c;e^{-t})}
=
\frac{a^c}{\sqrt{2ba^{c/b}}}
\exp\left(
\frac{\pi^2+3(2b)^{-1}\log^2 a}{6t}
+
\left(\frac{c^2}{4b}-\frac{1}{24}\right)t
+
O(e^{-{K}/{t}})
\right),
\end{align}
as $t\to 0^+$, where $K$ is a positive constant which generally depends on
$a$ and $b$, but not on $c$. We note that the case $b=1$ in
\eqref{eqmm10} was observed and suggested by McIntosh
\cite[p. 135]{MR1310726}.

We also define
\begin{align}\label{eqmm2}
f_2(a,b,c;q)
:=
\sum_{n\ge 0}a^nq^{bn^2+cn}(-q)_n
=
\sum_{n\ge 0}(aq^c)^nq^{bn^2}(-q)_n,
\end{align}
and
\begin{align*}
\tilde{f}_2(a,b,c;q)
:=
f_1\left(
a^{\frac{1}{1+2b}},
\frac{b}{1+2b},
\frac{c-b}{1+2b};
q
\right),
\end{align*}
where either $c\in\br$, $a>1/2$, and $b>0$, or $c\in\br$,
$1/2<a<1$, and $b=0$. McIntosh \cite[p. 421]{MR1618298} stated that
numerical computations suggest
\begin{align}\label{eqmm20}
\frac{f_2(a,b,c;e^{-t})}{\tilde{f}_2(a,b,c;e^{-t})}
=
\sqrt{\frac{\pi a^{-\frac{1+2c}{1+2b}}}{(1+2b)t}}
\exp\left(
\frac{\frac{6\log^2 a}{2b+1}-\pi^2}{12t}
+
\frac{6c^2+6c-b+1}{12(2b+1)}t
+
O(e^{-{K}/{t}})
\right),
\end{align}
as $t\to 0^+$, where $K$ is a positive constant which generally depends on
$a$ and $b$, but not on $c$.

\subsection{Bilateral series and McIntosh's conjectures}\label{sec12}

The first aim of this paper is to establish $q$-series
transformation formulas that will aid in proving the asymptotic formulas
\eqref{eqmm10} and \eqref{eqmm20}. It follows directly from the definition
of the $q$-shifted factorials (cf.\ \cite[Appendix (I.1) and (I.2)]{MR2128719})
that
\[
{1}/{(q)_{-n}}=0
\qquad \text{and} \qquad
(-q)_{-n}(-1)_n=q^{n(n-1)/2},
\]
for all positive integers $n$. Thus, the series \eqref{eqmm1} and
\eqref{eqmm2} admit representations as the following bilateral series,
respectively:
\begin{align*}
\hat{f}_1(a,b,c;q)
&:=
\sum_{n\in\bz}
\frac{(a q^{c+b})^n q^{2b\binom{n}{2}}}{(q)_n}
=
f_1(a,b,c;q), \\
\hat{f}_2(a,b,c;q)
&:=
\sum_{n\in\bz}(aq^c)^n q^{bn^2}(-q)_n \\
&=
\sum_{n\in\bz}
\frac{(a^{-1}q^{b-c})^n q^{(1+2b)\binom{n}{2}}}{(-1)_n}=
f_2(a,b,c;q)
+
\sum_{n\ge 1}
\frac{q^{bn^2-cn}}{a^n(-1;q^{-1})_n}.
\end{align*}
By an elementary argument; see Proposition \ref{prop10} in Section
\ref{sec4}, we prove that, for $c\in\mathbb{R}$ and either
$a>1/2$ and $b>0$, or $1/2<a<1$ and $b=0$, one has
\[
\hat{f}_2(a,b,c;e^{-t})
=
f_2(a,b,c;e^{-t})
\left(1+O(e^{-\delta/t})\right)
\quad \text{as} \quad t\to 0^+,
\]
where $\delta$ is a positive constant that generally depends on $a$ and
$b$, but not on $c$. Hence, to prove \eqref{eqmm20}, it suffices to prove
that \eqref{eqmm20} holds with $f_2(a,b,c;q)$ replaced by the bilateral
series $\hat{f}_2(a,b,c;q)$.

\medskip
Since $f_1(a,b,c;q)$ is already equal to the bilateral series
$\hat{f}_1(a,b,c;q)$, and since both series $\hat{f}_1$ and
$\hat{f}_2$ can be written in the same form, we introduce the following
general bilateral $q$-series. For $\alpha\ge 1$ and $x\in\bc$, define
\begin{align}\label{eqdefl}
L_\alpha(x;q,z)
=
\sum_{n\in\mathbb{Z}}
\frac{z^n q^{\alpha\binom{n}{2}}}{(x)_n},
\end{align}
for any $z\in\bc$ for which the series converges. By definition, the
bilateral series
$\hat{f}_1(a,b,c;q)$, $\tilde{f}_1(a,b,c;q)$,
$\hat{f}_2(a,b,c;q)$, and $\tilde{f}_2(a,b,c;q)$ all arise from
$L_\alpha(x;q,z)$ by choosing suitable parameters $\alpha$, $x$, and $z$.
In particular,
\begin{align*}
\hat{f}_1(a,b,c;q)
&=
L_{2b}\left(q;q,aq^{b+c}\right), \\
\tilde{f}_1(a,b,c;q)
&=
L_{1-\frac{1}{2b}}
\left(
q;q,
(-q)(aq^{b+c})^{-\frac{1}{b}}
q^{\frac12\left(1-\frac{1}{2b}\right)}
\right), \\
\hat{f}_2(a,b,c;q)
&=
L_{1+2b}\left(-1;q,a^{-1}q^{b-c}\right), \\
\tilde{f}_2(a,b,c;q)
&=
L_{1-\frac{1}{1+2b}}
\left(
q;q,
(a^{-1}q^{b-c})^{-\frac{1}{1+2b}}
q^{\frac12\left(1-\frac{1}{1+2b}\right)}
\right).
\end{align*}
Thus, McIntosh's conjectured formulas \eqref{eqmm10} and \eqref{eqmm20}
suggest that there should exist a transformation formula relating
$L_\alpha(x;q,z)$ to
\[
L_{1-1/\alpha}
\left(
q;q,
-xz^{-1/\alpha}q^{\frac12(1-1/\alpha)}
\right).
\]
Equivalently, this would be a transformation formula relating
\[
L_\alpha(zx;q,z^\alpha)
\quad \text{to} \quad
L_{1-1/\alpha}
\left(
q;q,
-xq^{\frac12(1-1/\alpha)}
\right).
\]
The first result of this paper establishes the existence of such a
transformation formula.
\begin{theorem}
  \label{main1}
Let $z>0$ and $q=e^{-t}$ with $t>0$. For any $\alpha>1$ and $x\in\bc$, we have
\begin{align*}
L_\alpha (zx; q, z^\alpha)
=
\frac{(2\pi z^\alpha)^{1/2}e^{\frac{\alpha t}{8}+\frac{\alpha}{2t}(\log z)^2}}{(\alpha t)^{1/2}(zx)_\infty}
\sum_{n\in\bz}
L_{1-1/\alpha}\left(q; q, -e^{2\pi \ri n/\alpha}xq^{\frac{1}{2}(1-\frac{1}{\alpha})}\right)
e^{2\pi \ri n(\frac{1}{2}-\frac{1}{t}\log z)-\frac{2\pi^2n^2}{\alpha t}}.
\end{align*}
\end{theorem}

It should be noted that if we further assume $|x|<1$, then Theorem \ref{main1}
also holds for $\alpha=1$. Moreover, the general term in the series on the
right-hand side of the transformation formula of Theorem \ref{main1} is bounded by
\[
e^{-\frac{2\pi^2n^2}{\alpha t}}
L_{1-1/\alpha}\left(q; q, |x|q^{\frac{1}{2}(1-\frac{1}{\alpha})}\right),
\]
and it decays exponentially as $t\to 0^+$. The series therefore converges
with the same precision as ordinary theta functions. This immediately implies
the following asymptotic formula.
\begin{corollary}\label{cor02}
Let $z>0$ and $q=e^{-t}$ with $t>0$. For any $x\in\bc$ and $\alpha>1$, we have
\begin{align}\label{eqmmmm}
\frac{(zx)_\infty e^{-\frac{\alpha t}{8}-\frac{\alpha}{2t}(\log z)^2}}
{\sqrt{2\pi z^\alpha/(\alpha t)}}\,
L_\alpha(zx;q,z^\alpha)
=&\,
L_{1-{1}/{\alpha}}
\left(q;q,-xq^{\frac{1}{2}(1-\frac{1}{\alpha})}\right)
\nonumber\\
&+
O\left(
e^{-\frac{2\pi^2}{\alpha t}}
L_{1-{1}/{\alpha}}
\left(q;q,|x|q^{\frac{1}{2}(1-\frac{1}{\alpha})}\right)
\right),
\end{align}
where the implied constant depends only on $\alpha$.
\end{corollary}

For real $x\le0$, we have $|x|=-x$. Hence, the big-$O$ term in the
corollary above is exponentially small relative to the main term
\[
L_{1-{1}/{\alpha}}
\left(q;q,-xq^{\frac{1}{2}(1-\frac{1}{\alpha})}\right).
\]
Moreover, using \eqref{eqdedk}, we have
\[
(-1)_\infty
=
2(-q)_\infty
=
2\frac{(q^2;q^2)_\infty}{(q;q)_\infty}
=
\sqrt{2}\,
e^{\frac{\pi^2}{12t}+\frac{t}{24}}
\left(1+O(e^{-2\pi^2/t})\right).
\]
Thus, by substituting $x\mapsto -1/z$ into Corollary \ref{cor02}, we obtain
a proof of McIntosh's conjecture \eqref{eqmm20}. We state this result as the
following corollary.
\begin{corollary}\label{cor2}
Let $z>0$ and $q=e^{-t}$ with $t>0$. For any $x\le 0$ and $\alpha>1$, we have
\begin{align*}
L_\alpha(zx;q,z^\alpha)
&=
\frac{(2\pi z^\alpha)^{1/2}e^{\frac{\alpha t}{8}+\frac{\alpha}{2t}(\log z)^2}}
{(\alpha t)^{1/2}(xz)_\infty}
L_{1-{1}/{\alpha}}
\left(q;q,-xq^{\frac{1}{2}(1-\frac{1}{\alpha})}\right)
\left(1+O\left(e^{-{2\pi^2}/{(\alpha t)}}\right)\right),
\end{align*}
where the implied constant depends only on $\alpha$. In particular,
\begin{align*}
\frac{L_\alpha(-1;q,z^\alpha)}
{L_{1-{1}/{\alpha}}\left(q;q,z^{-1}q^{\frac{1}{2}(1-{1}/{\alpha})}\right)}
=
\sqrt{\frac{\pi z^\alpha}{\alpha t}}
\exp\left(
\frac{6\alpha\log^2 z-\pi^2}{12t}
+\frac{3\alpha-1}{24}t
+O\left(e^{-{2\pi^2}/{(\alpha t)}}\right)
\right).
\end{align*}
In other words, McIntosh's conjecture \eqref{eqmm20} is true.
\end{corollary}
Corollary \ref{cor02} also provides a partial proof of McIntosh's conjecture
\eqref{eqmm10}, provided that the big-$O$ term in Corollary \ref{cor02} is
exponentially smaller than the left-hand side of \eqref{eqmmmm}. To see this,
let $w_\alpha$ and $w_\beta$ denote the positive roots of
$w+z^{-1}w^{1/\alpha}=1$ and $w+z^{-1}w^{1-1/\alpha}=1$, respectively.
Define $\delta_\alpha(z)$ by
\[
\delta_\alpha(z)
=
\frac{1}{\alpha}-\frac{1}{6}
+
\frac{
\li_2(w_\alpha)+\li_2(w_\beta)
+\frac{1}{2\alpha}\log^2 w_\alpha
+\frac{1}{2}\left(1-\frac{1}{\alpha}\right)\log^2 w_\beta
-(\log z)\log(w_\alpha w_\beta)
}{2\pi^2}.
\]
Clearly, the function $\delta_\alpha(z)$ is continuous in $\alpha$ and $z$
for $\alpha\ge 1$ and $z>0$, and $\delta_\alpha(1)>0$ for all
$1\le \alpha\le 6$. By combining Theorem \ref{macmain} of McIntosh with
Corollary \ref{cor02}, we obtain the following corollary; the details are
provided in Section \ref{sec4}.

\begin{corollary}\label{cor1}
Let $z>0$ and $q=e^{-t}$ with $t>0$. For any $\alpha>1$, we have
\begin{align*}
L_\alpha(q;q,z^\alpha)
&=
\frac{(2\pi z^\alpha)^{1/2}
e^{\frac{\alpha t}{8}+\frac{\alpha}{2t}(\log z)^2}}
{(\alpha t)^{1/2}(q)_\infty}
L_{1-1/\alpha}
\left(q;q,-z^{-1}q^{\frac{3}{2}-\frac{1}{2\alpha}}\right)
\left(1+O\left(e^{-{2\pi^2}\delta_\alpha(z)/t}\right)\right),
\end{align*}
where the implied constant depends only on $\alpha$ and $z$. In particular,
\begin{align*}
\frac{L_\alpha(q;q,z^\alpha)}
{L_{1-1/\alpha}
\left(q;q,-z^{-1}q^{\frac{3}{2}-\frac{1}{2\alpha}}\right)}
&=
\frac{z^{\alpha/2}}{\sqrt{\alpha}}
\exp\left(
\frac{\pi^2+3\alpha\log^2 z}{6t}
+\frac{3\alpha-1}{24}t
+O\left(e^{-\frac{2\pi^2}{t}\min(2,\delta_\alpha(z))}\right)
\right).
\end{align*}
\end{corollary}

We note that, in deriving the second asymptotic formula in Corollary
\ref{cor1}, we used the transformation formula \eqref{eqdedk} for
$(q)_\infty$. Moreover, since the conjectured formula \eqref{eqmm10} concerns
the limit $q\to1$, we apply Corollary \ref{cor1} with
$\alpha=2b$ and $z=(aq^{b+c})^{1/\alpha}$. The positivity condition
\[
\delta_{2b}\left(a^{\frac{1}{2b}}\right)
=
\lim_{q\to1}
\delta_{2b}
\left(a^{\frac{1}{2b}}q^{\frac{1}{2}+\frac{c}{2b}}\right)
=
\lim_{q\to1}\delta_\alpha(z)
>0,
\]
which is justified by the continuity of $\delta_\alpha(z)$ for
$\alpha\ge1$ and $z>0$, implies \eqref{eqmm10}. This gives a partial answer
to McIntosh's conjecture \eqref{eqmm10}.

\begin{corollary}
For all $b>1/2$ and $a>0$ such that
$\delta_{2b}(a^{1/(2b)})>0$, McIntosh's conjecture \eqref{eqmm10} holds.
\end{corollary}

\subsection{Theta functions and bilateral basic hypergeometric series}

In fact, we derive analogous transformation formulas for the following
bilateral $q$-series $L_\alpha(\bm{x};\bm{y};q,z)$, which is more general than
$L_\alpha(x;q,z)$ defined in \eqref{eqdefl}. Let $r$ be a positive integer,
let $\alpha\ge r$ be real, and let $\bm{x}=(x_1,x_2,\ldots,x_r),
\bm{y}=(y_1,y_2,\ldots,y_r)
$,
where all $x_s$ and $y_s$ are nonzero complex numbers. We define
\begin{align}\label{mdef}
L_\alpha(\bm{x};\bm{y};q,z)
&:=
\sum_{n\in\bz}z^{n} q^{(\alpha-r)\binom{n}{2}}
\prod_{1\le s\le r}
\frac{(1/y_s)_n(-y_s)^n}{(x_s)_n}
\nonumber\\
&=
\frac{1}{\prod_{1\le s\le r}(x_s,qy_s)_\infty}
\sum_{n\in\bz}
z^{n}q^{\alpha\binom{n}{2}}
\prod_{1\le s\le r}
(x_sq^{n},y_s q^{1-n})_\infty
\end{align}
for any $z\in\bc$ for which the series converges. Here, in the second equality of \eqref{mdef}, we have used the
\emph{reflection relation} for $q$-shifted factorials
(cf.\ \cite[Appendix, (I.2)]{MR2128719}), namely
\begin{align}\label{eqrf}
(1/u)_n(uq)_{-n}=(-u)^{-n}q^{\binom{n}{2}},
\end{align}
for any $u\neq 0$. Since the parameter $\alpha\ge r$ can be chosen arbitrarily, the series
$L_\alpha(\bm{x};\bm{y};q,z)$ is more general than the classical bilateral
basic hypergeometric series ${}_r\psi_r$ (cf.\ \cite{MR2128719}). The latter is
formally defined, for any positive integer $r$, complex numbers $x_1,y_1,x_2,y_2,\ldots,x_r,y_r$ and $z$ by
\[
{}_r\psi_r
\left(
\begin{array}{c}
y_1,\ldots,y_r \\
x_1,\ldots,x_r
\end{array}
;q,z
\right)
=
\sum_{n\in\mathbb{Z}}
\frac{(y_1,\ldots,y_r)_n}{(x_1,\ldots,x_r)_n}z^n.
\]
In particular,
\[
L_r(\bm{x};\bm{y};q,z)
=
{}_r\psi_r
\left(
\begin{array}{c}
1/y_1,\ldots,1/y_r \\
x_1,\ldots,x_r
\end{array}
;q,(-1)^r y_1y_2\cdots y_r z
\right).
\]

We further define the multiple $q$-series $H_{\alpha}(\bm{x};\bm{y};q)$ by
\begin{align}\label{eqdef18}
H_{\alpha}(\bm{x};\bm{y};q)
:=
\sum_{\substack{i_s,j_s\ge 0\\ 1\le s\le r}}
\frac{
q^{\frac{1}{2}Q_{\alpha}(\bm{i},\bm{j})}
\prod_{1\le s\le r}(-x_s)^{i_s}(-y_s)^{j_s}
}{
\prod_{1\le s\le r}(q)_{i_s}(q)_{j_s}
},
\end{align}
where $Q_{\alpha}(\bm{i},\bm{j})$ is the quadratic form defined by
\begin{align*}
Q_{\alpha}(\bm{i},\bm{j})
:=
\sum_{1\le s\le r}(i_s^2+j_s^2)
-
\frac{1}{\alpha}
\left(\sum_{1\le s\le r}(i_s-j_s)\right)^2.
\end{align*}
By elementary arguments, the quadratic form
$Q_{\alpha}(\bm{i},\bm{j})$ is positive definite when $\alpha>r$ and
positive semidefinite when $\alpha=r$. Consequently, $H_{\alpha}(\bm{x};\bm{y};q)$
is an entire function of $\bm{x}$ and $\bm{y}$ when $\alpha>r$, while for
$\alpha=r$ it is analytic in these variables $|x_1|,|y_1|, |x_2|, |y_2|,\ldots, |x_r|, |y_r|<1$.

\medskip
The first main result of this paper is the following transformation formula.

\begin{theorem}\label{mth1}
Let $q=e^{-t}$ with $t>0$. For any real number $\alpha>r$, we have
\begin{align*}
L_\alpha(z\bm{x};z^{-1}\bm{y};q,z^\alpha)
&=
\frac{
\left(\frac{2\pi z^\alpha}{\alpha t}\right)^{1/2}
e^{\frac{\alpha t}{8}+\frac{\alpha}{2t}(\log z)^2}
}{
\prod_{1\le s\le r}(zx_s,qz^{-1}y_s)_\infty
}
\\
&\quad \times
\sum_{n\in\bz}
H_{\alpha}
\left(
e^{2\pi\ri n/\alpha}\bm{x};
e^{-2\pi\ri n/\alpha}\bm{y};
q
\right)
e^{
2\pi\ri n\left(\frac{1}{2}-\frac{1}{t}\log z\right)
-\frac{2\pi^2n^2}{\alpha t}
}.
\end{align*}
\end{theorem}

Note that, if we further assume that
$|x_s|,|y_s|<1$ for all $1\le s\le r$, then Theorem \ref{mth1} remains valid
for $\alpha=r$. Letting $r=1$, $x_1=x$, and $y_1\to0$ in the definitions
\eqref{mdef} and \eqref{eqdef18}, we find, after a simple algebraic
manipulation, that
\begin{align}\label{eqmll}
\begin{split}
L_\alpha(x;0;q,z)
&=
L_\alpha(x;q,z),\\
H_{\alpha}(x;0;q)
&=
\sum_{i\ge0}
\frac{
q^{\frac{1}{2}(1-1/\alpha)i^2}(-x)^i
}{(q)_i}
=
L_{1-1/\alpha}
\left(q;q,-xq^{\frac{1}{2}(1-1/\alpha)}\right).
\end{split}
\end{align}
Thus, Theorem \ref{main1} follows immediately from Theorem \ref{mth1}.

Moreover, for non-positive real numbers
$x_1,y_1,x_2,y_2,\ldots,x_r,y_r$, an argument similar to that used in
Corollary \ref{cor2} gives the following asymptotic transformation from
Theorem \ref{mth1}.

\begin{corollary}
Let $x_s,y_s\le0$ for all $1\le s\le r$, let $z>0$, and let
$q=e^{-t}$ with $t\to0^+$. Then, for any real number $\alpha>r$, we have
\begin{align*}
\frac{
L_\alpha(z\bm{x};z^{-1}\bm{y};q,z^\alpha)
}{
H_{\alpha}(\bm{x};\bm{y};q)
}
&=
\frac{
\left(\frac{2\pi z^\alpha}{\alpha t}\right)^{1/2}
e^{\frac{\alpha t}{8}+\frac{\alpha}{2t}(\log z)^2}
}{
\prod_{1\le s\le r}(zx_s,qz^{-1}y_s)_\infty
}
\left(
1+O\left(e^{-{2\pi^2}/{(\alpha t)}}\right)
\right),
\end{align*}
where the implied constant depends only on $\alpha$.
\end{corollary}

\medskip
For $\alpha=a/b$, where $a$ and $b$ are coprime positive integers, since
$e^{2\pi\ri n/\alpha}$ is $a$-periodic in $n$, the functions $H_{\alpha}\left(e^{2\pi\ri n/\alpha}\bm{x};e^{-2\pi\ri n/\alpha}\bm{y};q\right)$
appearing in Theorem \ref{mth1} repeat as $n$ ranges over $\mathbb{Z}$.
This suggests the existence of a more elegant transformation formula for the
bilateral $q$-series $L_\alpha(z\bm{x};z^{-1}\bm{y};q,z^\alpha)$. We define the \emph{theta function} $\theta(z;q)$ for any $z\in\bc$ by
\begin{equation}\label{eqmtheta}
\theta(z;q):=\sum_{n\in\bz}(-z)^n q^{\binom{n}{2}}=(z,q/z,q)_\infty.
\end{equation}
Here, the second equality follows from the Jacobi triple product identity.
Let $\zeta_a^k:=e^{2\pi\ri k/a}$ denote an $a$th root of unity for any
$k\in\bz$. Then our second main result is the following transformation formula.

\begin{theorem}\label{mth}
For any coprime positive integers $a$ and $b$ such that $a\ge br$, we have
\begin{align*}
L_{a/b}(z^b\bm{x};z^{-b}\bm{y};q,z^a)
=
\frac{1}{a}
\sum_{u\,(\bmod a)}
\frac{H_{a/b}(\zeta_a^{u}\bm{x};\zeta_a^{-u}\bm{y};q)}
{\prod_{1\le s\le r}(z^bx_s,qz^{-b}y_s)_\infty}
\theta\!\left(-\zeta_a^{-u}zq^{\frac{1-a}{2ab}};q^{\frac{1}{ab}}\right),
\end{align*}
provided that $H_{a/b}(\zeta_a^{u}\bm{x};\zeta_a^{-u}\bm{y};q)$ absolute converges.
\end{theorem}

\begin{remark}
Theorem \ref{mth} shows that, for $\alpha\in\bq$, the series
$L_\alpha(\bm{x};\bm{y};q,z^\alpha)$ exhibits additional arithmetic
properties. Moreover, if we define
\begin{align*}
f(z)
&:=L_{a/b}(z^b\bm{x};z^{-b}\bm{y};q,z^a)
\prod_{1\le s\le r}(z^bx_s,qz^{-b}y_s)_\infty\\
&=\sum_{n\in\bz}z^{an} q^{\frac{a}{b}\binom{n}{2}}
\prod_{1\le s\le r}(z^bx_sq^n,z^{-b}y_sq^{1-n})_\infty,
\end{align*}
where the second equality follows from \eqref{mdef}. Then one sees that
$f(z)$ is analytic for all $z\in\bc-\{0\}$ and satisfies the
quasi-periodicity relation
\[
f(z)=z^a f(q^{1/b}z).
\]
\end{remark}
 
We are particularly interested in two special cases of Theorem \ref{mth}.
One such case occurs when $r=1$ and $\alpha$ is a rational number satisfying
$\alpha\ge 1$; this case constitutes a generalization of the classical
bilateral basic hypergeometric series ${}_1\psi_1$. The other occurs when
$\alpha=2$, in which case the transformation formula in Theorem \ref{mth}
has only two terms on its right-hand side. By choosing appropriate parameters
so that one term vanishes, we obtain a simpler transformation formula.

\medskip

For the case $r=1$ and $\alpha=b/a\in\bq$, we have the following corollaries. 
\begin{corollary}\label{corm1}
For any coprime positive integers $a$ and $b$ such that $a\ge b$, we have
\begin{align*}
L_{a/b}(z^bx; z^{-b}y; q, z^{a})=\frac{1}{a}\sum_{u~(\bmod a)}\frac{H_{a/b}(\zeta_a^{u}x;\zeta_a^{-u}y;q)}{(z^bx,qz^{-b}y)_\infty}\theta\left(-\zeta_a^{-u}zq^{\frac{1}{2ab}(1-a)}; q^{\frac{1}{ab}}\right),
\end{align*}
provided that $H_{a/b}(\zeta_a^{u}x;\zeta_a^{-u}y;q)$ absolute converges. In particular,
\begin{align}\label{eqr11}
L_{1}(zx; z^{-1}y; z, q)=\frac{ \theta\!\left(-z; q\right)}{(zx,qz^{-1}y)_\infty}H_1(x;y;q),
\end{align}
and
\begin{align}\label{eqr22}
L_{2}(zx; z^{-1}y; q, z^{2})=\frac{\theta\left(-zq^{-{1}/{4}}; q^{{1}/{2}}\right)H_{2}(x;y;q)+\theta\left(zq^{-{1}/{4}}; q^{{1}/{2}}\right)H_{2}(-x;-y;q)}{2(zx,qz^{-1}y)_\infty}.
\end{align}
\end{corollary}
The identity \eqref{eqr11} in fact is equvlaent the well-known Ramanujan $_1\psi_1$ summation formula, see Corollary \ref{corrama11} in Section \ref{sec3} for the details. 
Based on the identity \eqref{eqr22}, we further prove the following summation formulas in Section \ref{sec3}.
\begin{corollary}\label{cormm}We have
\begin{align*}
\sum_{n\in\bz}\frac{(-z/x)_n(xz)^n q^{\binom{n}{2}}}{(zx)_n}=\frac{(-x^2q;q^2)_\infty\theta(-z^2;q^2)}{(zx,-qz^{-1}x)_\infty},
\end{align*}
and
\begin{align*}
\sum_{n\in\bz}\frac{(z/x;q^2)_n(-xz)^n q^{n(n+1)}}{(qzx;q^2)_n}=\frac{\theta\left(-zq^{{1}/{2}}; q\right)(xq^{1/2})_\infty+\theta\left(zq^{{1}/{2}}; q\right)(-xq^{1/2})_\infty}{2(qzx,q^{2}z^{-1}x;q^2)_\infty}.
\end{align*}
\end{corollary}
We note that the first identity in Corollary \ref{cormm} is equivalent to the following bilateral extension of the Lebesgue identity
\begin{equation*}
\sum_{k\in\bz}\frac{(a)_k}{(bq)_k}q^{\binom{k+1}2}b^k=
\frac{(q^2,abq,q/ab,bq^2/a;q^2)_\infty}{(bq,q/a;q)_\infty},
\end{equation*} 
see Schlosser \cite[Identity (2.7)]{MR4632568} for details. The second identity seems to be new. 
\medskip

Letting $y=0$ in the main identity of Corollary \ref{corm1}, we obtain the following transformation formula, which actually provides further arithmetic insight into Theorem \ref{main1} and McIntosh's conjectures mentioned above. 

\begin{corollary}\label{corm2}For any coprime positive integers $a$ and $b$ such that $a> b$, we have
\begin{align*}
L_{b/a}(z^bx; q, z^{a})=\frac{1}{a}\sum_{u~(\bmod a)}\frac{\theta\big(-\zeta_a^{-u}zq^{\frac{1}{2ab}(1-a)}; q^{\frac{1}{ab}}\big)}{(z^bx)_\infty } L_{1-b/a}\left(q; q, -\zeta_a^{u}x q^{\frac{1}{2}(1-b/a)}\right).
\end{align*}
In particular,
\begin{align*}
\sum_{n\in\bz}\frac{z^{2n}q^{n^2}}{(qxz)_n}=\frac{\theta\left(-zq^{{3}/{4}}; q^{{1}/{2}}\right)}{2(qxz)_\infty}\sum_{n\ge 0}\frac{(-x)^n q^{\frac{1}{2}\binom{n+1}{2}} }{(q)_n}+\frac{\theta\left(zq^{{3}/{4}}; q^{{1}/{2}}\right)}{2(qxz)_\infty}\sum_{n\ge 0}\frac{x^n q^{\frac{1}{2}\binom{n+1}{2}} }{(q)_n}.
\end{align*}
\end{corollary}
Setting $x=1/z$ in the second identity of Corollary \ref{corm2}, we obtain
an equivalent form of an identity claimed by Ramanujan
(cf.\ \cite[Entry 7.2.4]{MR2474043}):
\begin{align}\label{eqrr1}
\sum_{n\ge 0}\frac{z^{2n}q^{n^2}}{(q)_n}
&=
\frac{\theta\left(-zq^{3/4};q^{1/2}\right)}{2(q)_\infty}
\sum_{n\ge 0}
\frac{(-z)^{-n}q^{\frac{1}{2}\binom{n+1}{2}}}{(q)_n}+
\frac{\theta\left(zq^{3/4};q^{1/2}\right)}{2(q)_\infty}
\sum_{n\ge 0}
\frac{z^n q^{\frac{1}{2}\binom{n+1}{2}}}{(q)_n},
\end{align}
which was proved by Andrews in \cite[Theorem 2.4]{MR557539}; see also
Dixit and Kumar \cite[Equation (6.1)]{MR4966569} for details. Thus, the
second identity of Corollary \ref{corm2} is a bilateral extension of
Ramanujan's identity \eqref{eqrr1}.

\medskip
For the case $\alpha=2$, we see that $r\in\{1,2\}$. In view of
Corollary \ref{corm1}, it remains only to discuss the case $r=2$.
Since $L_2(\bm{x};\bm{y};z,q)$ can be expressed as a ${}_2\psi_2$ sum,
Theorem \ref{mth} gives the following transformation formula.

\begin{corollary}\label{mcor}
We have
\begin{align*}
{}_2\psi_2 \!\left(
\begin{array}{c}
z/y_1,z/y_2 \\
zx_1,zx_2
\end{array};
q,y_1y_2
\right)
=
\frac{
H_{2}(\bm{x};\bm{y};q)
\theta\left(-zq^{-1/4};q^{1/2}\right)
+
H_{2}(-\bm{x};-\bm{y};q)
\theta\left(zq^{-1/4};q^{1/2}\right)
}
{
2\prod_{1\le s\le 2}(zx_s,qz^{-1}y_s)_\infty
}.
\end{align*}
In particular, for any $\ell\in\bz$, we have
\begin{align}\label{eqm3}
{}_2\psi_2 \!\left(
\begin{array}{c}
1/y_1,1/y_2 \\
x_1,x_2
\end{array};
q,q^{1/2+\ell}y_1y_2
\right)
=
\frac{
\theta\left(-q^{\ell/2};q^{1/2}\right)
H_{2}(q^{-1/4-\ell/2}\bm{x};
q^{1/4+\ell/2}\bm{y};q)
}
{
2\prod_{1\le s\le 2}(x_s,qy_s)_\infty
}.
\end{align}
\end{corollary}

Since $\theta(q^{\ell/2};q^{1/2})=0$ for all $\ell\in\bz$, identity
\eqref{eqm3} is obtained from the main identity of Corollary \ref{mcor}
by making the substitutions
\[
z\mapsto q^{1/4+\ell/2},\qquad
\bm{y}\mapsto q^{1/4+\ell/2}\bm{y},\qquad
\bm{x}\mapsto q^{-1/4-\ell/2}\bm{x}.
\]
Taking $x_1,y_1,y_2\to 0$, $x_2\mapsto q$, and $q\mapsto q^2$ in
\eqref{eqm3}, we obtain
\begin{equation}\label{eqgmit}
\sum_{n\ge 0}
\frac{q^{2n(n-1)+(1/2+\ell)n}}{(q^2;q^2)_n}
=
\frac{\theta\left(-q^\ell;q\right)}
{2(q^2;q^2)_\infty}
\sum_{n\ge 0}
\frac{(-1)^n q^{n^2/2+(3/2-\ell)n}}{(q^2;q^2)_n}.
\end{equation}
Using the identity
\[
\theta(z;q)=(-z)^\ell q^{\binom{\ell}{2}}\theta(q^\ell z;q),
\]
for all $ \ell\in\bz$, we have
\[
\theta\left(-q^\ell;q\right)
=
q^{-\binom{\ell}{2}}\theta(-1;q)
=
2q^{-\ell(\ell-1)/2}(-q)_\infty(q^2;q^2)_\infty.
\]
It follows that the identity \eqref{eqgmit} is an equivalent form of the following
identity of McIntosh \cite[p.~418, (10)]{MR1618298}:
\begin{equation}\label{mci}
\sum_{n\ge 0}
\frac{q^{2n^2+(2m+1)n+\frac{m(m+1)}{2}}}{(q^2;q^2)_n}
=
(-q)_\infty
\sum_{j\ge 0}
\frac{(-1)^j q^{j(j+1)/2-mj}}{(q^2;q^2)_j},
\end{equation}
for all $m\in\bz$. Thus Corollary \ref{mcor} generalizes McIntosh's identity \eqref{mci}.

\medskip
The rest of the paper is organized as follows. In Section \ref{sec2}, we
prove Theorems \ref{mth1} and \ref{mth}. The proofs rely on the
$q$-binomial theorem and on certain transformation and summation formulas
for theta functions. In Section \ref{sec3}, we present some applications of
Theorem \ref{mth}. In particular, we investigate the consequences of
Theorem \ref{mth} in the cases $\alpha=r=1$ and $\alpha=r=2$. We conclude
in Section \ref{sec4} with an asymptotic analysis that supports the
explanations given in this paper.

\section{The proof of the main theorems}\label{sec2}
In this section, we prove our main theorems. We first establish two basic
facts about the theta function $\theta(z;q)$, which will be used in the
proofs.

\subsection{Theta functions}

We begin with the following transformation property.

\begin{lemma}\label{lem21}
Let $z,t>0$. For any $\mu\in\br$, we have
\begin{align*}
\sum_{n\in \mu+\bz} z^n e^{-t\binom{n}{2}}
=
\left(\frac{2\pi z}{t}\right)^{1/2}
e^{\frac{t}{8}+\frac{1}{2t}(\log z)^2}
\sum_{n\in\bz}
e^{2\pi \ri n\left(\mu-\frac{1}{2}-\frac{1}{t}\log z\right)
-\frac{2\pi^2 n^2}{t}}.
\end{align*}
\end{lemma}

\begin{proof}
Taking the Fourier series expansion of the $1$-periodic function
\[
\sum_{n\in\bz} e^{-(n+x)^2t}, 
\]
for $x\in\br$ yields the well-known formula
\[
\sum_{n\in\bz} e^{-(n+x)^2t}
=
\left(\frac{\pi}{t}\right)^{1/2}
\sum_{n\in\bz} e^{-\pi^2 n^2/t+2\pi\ri nx}.
\]
After a simple algebraic manipulation, we obtain
\begin{align*}
\sum_{n\in \mu+\bz} z^n e^{-t\binom{n}{2}}
&=
e^{\frac{t}{2}\left(\frac{1}{2}+t^{-1}\log z\right)^2}
\sum_{n\in \mu+\bz}
e^{-\frac{t}{2}
\left[n^2-2n\left(\frac{1}{2}+t^{-1}\log z\right)
+\left(\frac{1}{2}+t^{-1}\log z\right)^2\right]} \\
&=
z^{1/2}e^{\frac{t}{8}+\frac{1}{2t}(\log z)^2}
\sum_{n\in\bz}
e^{-\frac{t}{2}
\left(n+\mu-\frac{1}{2}-t^{-1}\log z\right)^2} \\
&=
\left(\frac{2\pi z}{t}\right)^{1/2}
e^{\frac{t}{8}+\frac{1}{2t}(\log z)^2}
\sum_{n\in\bz}
e^{-\frac{2\pi^2 n^2}{t}
+2\pi\ri n\left(\mu-\frac{1}{2}-t^{-1}\log z\right)}.
\end{align*}
This completes the proof.
\end{proof}

We also need the following summation formula.

\begin{lemma}\label{lem22}
Let $a,b\in\bn$ with $\gcd(a,b)=1$, and let $u\in\bz$. Then
\[
\sum_{v~(\bmod a)}
\zeta_a^{-ubv}
\sum_{n\in bv/a+\bz}
z^{an}q^{a\binom{n}{2}}
=
\theta\left(
-z\zeta_a^{-u}q^{\frac{1-a}{2a}};
q^{1/a}
\right).
\]
\end{lemma}

\begin{proof}
Since $\gcd(a,b)=1$, the integers $bv$ run over a complete residue system
modulo $a$ as $v$ runs over a complete residue system modulo $a$. Thus,
\begin{align*}
\sum_{v~(\bmod a)}
\zeta_a^{-ubv}
\sum_{n\in bv/a+\bz}
z^{an}q^{a\binom{n}{2}}
&=
\sum_{v~(\bmod a)}
\sum_{n\in bv/a+\bz}
(z\zeta_a^{-u})^{an}q^{a\binom{n}{2}} \\
&=
\sum_{n\in \frac{1}{a}\bz}
(z\zeta_a^{-u})^{an}q^{a\binom{n}{2}} \\
&=
\sum_{n\in\bz}
\left(z\zeta_a^{-u}q^{\frac{1-a}{2a}}\right)^n
q^{\frac{1}{a}\binom{n}{2}},
\end{align*}
that is
$$\sum_{v~(\bmod a)}
\zeta_a^{-ubv}
\sum_{n\in bv/a+\bz}
z^{an}q^{a\binom{n}{2}}=
\theta\left(
-z\zeta_a^{-u}q^{\frac{1-a}{2a}};
q^{1/a}
\right),$$
which completes the proof.
\end{proof}

\subsection{Proofs of Theorems \ref{mth1} and \ref{mth}}

We now prove our main results. We begin by deriving the transformation
formula \eqref{eq21} below for $L_\alpha(\bm{x};\bm{y};q,z)$, which is
related to the theta-type series
\[
\sum_{n\in\mu+\mathbb{Z}} z^n q^{\binom{n}{2}}
\]
studied in the previous subsection. Recall from \eqref{mdef} that
\begin{align*}
L_\alpha(\bm{x};\bm{y};q,z)
=
\frac{1}{\prod_{1\le s\le r}(x_s)_\infty(qy_s)_\infty}
\sum_{n\in\bz}
z^n q^{\alpha\binom{n}{2}}
\prod_{1\le s\le r}
(x_sq^n)_\infty(y_s q^{1-n})_\infty .
\end{align*}
Expanding the products in the summand into series, we obtain
\begin{align*}
L_\alpha(\bm{x};\bm{y};q,z^\alpha)
&=
\sum_{\substack{i_s,j_s\ge 0\\ 1\le s\le r}}
\frac{
q^{\sum_{1\le s\le r}
\left(\binom{i_s}{2}+\binom{j_s+1}{2}\right)}
\prod_{1\le s\le r}(-x_s)^{i_s}(-y_s)^{j_s}
}
{
\prod_{1\le s\le r}(q)_{i_s}(q)_{j_s}
}
\\
&\quad\times
\frac{
\sum_{n\in\mathbb{Z}}
z^{\alpha n}
q^{\alpha\binom{n}{2}
+n\sum_{1\le s\le r}(i_s-j_s)}
}
{
\prod_{1\le s\le r}(x_s,qy_s)_\infty
}.
\end{align*}
Here we have used Euler's well-known identities
\[
(x)_\infty
=
\sum_{h\ge 0}
\frac{(-x)^h q^{\binom{h}{2}}}{(q)_h}
\;\; \text{and}\;\;
(xq)_\infty
=
\sum_{h\ge 0}
\frac{(-x)^h q^{\binom{h+1}{2}}}{(q)_h}.
\]
Therefore, after a simple algebraic manipulation, we obtain
\begin{align}\label{eq21}
L_\alpha(z\bm{x};z^{-1}\bm{y};q,z^\alpha)
&=
\sum_{\substack{i_s,j_s\ge 0\\ 1\le s\le r}}
\frac{
q^{\frac{1}{2}Q_\alpha(\bm{i},\bm{j})}
\prod_{1\le s\le r}(-x_s)^{i_s}(-y_s)^{j_s}
}
{
\prod_{1\le s\le r}
(q)_{i_s}(q)_{j_s}(zx_s,qz^{-1}y_s)_\infty
}
\notag\\
&\quad\times
\sum_{n\in \frac{1}{\alpha}
\sum_{1\le s\le r}(i_s-j_s)+\bz}
z^{\alpha n}q^{\alpha\binom{n}{2}},
\end{align}
where $Q_\alpha(\bm{i},\bm{j})$ is defined in \eqref{eqdef18}.

\begin{proof}[Proof of Theorem \ref{mth1}]
Applying Lemma \ref{lem21} with
\[
\mu=\frac{1}{\alpha}\sum_{1\le s\le r}(i_s-j_s)
\]
to the inner sum in \eqref{eq21}, we obtain
\begin{align*}
&\frac{\prod_{1\le s\le r}(zx_s,qz^{-1}y_s)_\infty}
{(2\pi z^\alpha/(\alpha t))^{1/2}
e^{\frac{\alpha t}{8}+\frac{1}{2\alpha t}(\log z^\alpha)^2}}
L_\alpha(z{\bm x};z^{-1}{\bm y};q,z^\alpha) \\
&=
\sum_{\substack{i_s,j_s\ge 0\\ 1\le s\le r}}
\frac{q^{\frac12 Q_\alpha(\bm{i},\bm{j})}
\prod_{1\le s\le r}(-x_s)^{i_s}(-y_s)^{j_s}}
{\prod_{1\le s\le r}(q)_{i_s}(q)_{j_s}}
\sum_{n\in\mathbb{Z}}
e^{-\frac{2\pi^2 n^2}{\alpha t}
+2\pi\ri n\left(\frac{1}{\alpha}\sum_{1\le s\le r}(i_s-j_s)+\frac12-\frac{1}{\alpha t}\log z^\alpha\right)} \\
&=
\sum_{n\in\mathbb{Z}}
H_\alpha\!\left(e^{2\pi\ri n/\alpha}{\bm x};
e^{-2\pi\ri n/\alpha}{\bm y};q\right)
e^{2\pi\ri n\left(\frac12-\frac{1}{t}\log z\right)-\frac{2\pi^2 n^2}{\alpha t}},
\end{align*}
where we have used the definition \eqref{eqdef18} of
$H_\alpha({\bm x};{\bm y};q)$.
This completes the proof.
\end{proof}

\begin{proof}[Proof of Theorem \ref{mth}]
Recall that $a$ and $b$ are coprime positive integers. Since $b$ is invertible
modulo $a$, if $u$ runs through a complete residue system modulo $a$, then so
does $bu$. Therefore,
\[
\frac{1}{a}\sum_{u~(\bmod a)}\zeta_a^{-buv}\zeta_a^{bun}
=
\begin{cases}
1, & n\equiv v \pmod a,\\
0, & n\not\equiv v \pmod a.
\end{cases}
\]
Using this orthogonality relation, together with the definition
\eqref{eqdef18} of $H_\alpha({\bm x};{\bm y};q)$, we obtain
\begin{align*}
\frac{1}{a}\sum_{u~(\bmod a)}\zeta_a^{-buv}
H_\alpha(\zeta_a^{bu}{\bm x};\zeta_a^{-bu}{\bm y};q)
&=
\sum_{\substack{i_s,j_s\ge 0\\ 1\le s\le r\\
\sum_{1\le s\le r}(i_s-j_s)\equiv v~(\bmod a)}}
\frac{q^{\frac12 Q_\alpha(\bm{i},\bm{j})}
\prod_{1\le s\le r}(-x_s)^{i_s}(-y_s)^{j_s}}
{\prod_{1\le s\le r}(q)_{i_s}(q)_{j_s}}.
\end{align*}
Hence, by splitting the sum in \eqref{eq21} according to the residue class of
$\sum_{1\le s\le r}(i_s-j_s)$ modulo $a$, we obtain
\begin{align*}
L_\alpha\!\left(z^b{\bm x};z^{-b}{\bm y};q,z^a\right)
&=
\frac{1}{\prod_{1\le s\le r}(z^bx_s,qz^{-b}y_s)_\infty}
\sum_{v~(\bmod a)}
\sum_{n\in \frac{vb}{a}+\bz}
z^{an}q^{\alpha\binom{n}{2}} \\
&\qquad\times \sum_{\substack{i_s,j_s\ge 0\\ 1\le s\le r\\
\sum_{1\le s\le r}(i_s-j_s)\equiv v~(\bmod a)}}
\frac{q^{\frac12 Q_\alpha(\bm{i},\bm{j})}
\prod_{1\le s\le r}(-x_s)^{i_s}(-y_s)^{j_s}}
{\prod_{1\le s\le r}(q)_{i_s}(q)_{j_s}}\\
=&\frac{1}{\prod_{1\le s\le r}(z^bx_s,qz^{-b}y_s)_\infty}
\sum_{v~(\bmod a)}
\sum_{n\in \frac{vb}{a}+\bz}
z^{an}q^{\alpha\binom{n}{2}}\\
&\qquad\times\frac{1}{a}\sum_{u~(\bmod a)}\zeta_a^{-buv}
H_\alpha(\zeta_a^{bu}{\bm x};\zeta_a^{-bu}{\bm y};q).
\end{align*}
Applying Lemma \ref{lem22} to the identity above, we obtain
\begin{align*}
L_\alpha\!\left(z^b{\bm x};z^{-b}{\bm y};q,z^a\right)
&=
\frac{1}{a}\sum_{u~(\bmod a)}
\frac{H_\alpha(\zeta_a^u{\bm x};\zeta_a^{-u}{\bm y};q)}
{\prod_{1\le s\le r}(z^bx_s,qz^{-b}y_s)_\infty} 
\sum_{v~(\bmod a)}\zeta_a^{-ubv}
\sum_{n\in \frac{bv}{a}+\bz}
z^{an}q^{\frac{a}{b}\binom{n}{2}} \\
&=
\frac{1}{a}\sum_{u~(\bmod a)}
\frac{H_\alpha(\zeta_a^u{\bm x};\zeta_a^{-u}{\bm y};q)}
{\prod_{1\le s\le r}(z^bx_s,qz^{-b}y_s)_\infty}
\theta\!\left(-z\zeta_a^{-u}q^{\frac{1-a}{2ab}};q^{\frac{1}{ab}}\right),
\end{align*}
which completes the proof.
\end{proof}

\section{Applications of Theorem \ref{mth}}\label{sec3}
Some special cases of Theorem \ref{mth} yield transformation formulas
relating a basic hypergeometric series to a multiple basic hypergeometric
series. From these formulas, we can further derive some summation formulas.

\subsection{On Ramanujan's ${}_1\psi_1$ summation}

Setting $\alpha=r=1$ in Theorem \ref{mth}, we obtain the famous
Ramanujan ${}_1\psi_1$ summation formula.

\begin{corollary}\label{corrama11}
We have
\begin{align*}
L_1(zx;z^{-1}y;q,z)
=
\frac{\theta(z;q)(xy)_\infty}
{(zx,qz^{-1}y,-x,-y)_\infty}.
\end{align*}
\end{corollary}

\begin{proof}
Setting $\alpha=r=1$ in Theorem \ref{mth}, we obtain
\begin{align*}
L_1(zx;z^{-1}y;q,z)
&=
\frac{\theta(z;q)}{(zx,qz^{-1}y)_\infty}H_1(x;y;q)\\
&=
\frac{\theta(z;q)}{(zx,qz^{-1}y)_\infty}
\sum_{i,j\ge 0}
\frac{(-x)^i(-y)^j q^{ij}}{(q)_i(q)_j}\\
&=
\frac{\theta(z;q)}{(zx,qz^{-1}y)_\infty}
\sum_{i\ge 0}
\frac{(-x)^i}{(q)_i}\frac{1}{(-yq^i)_\infty}.
\end{align*}
Here we have used Euler's well-known identity
\[
\sum_{n\ge 0}\frac{t^n}{(q)_n}
=
\frac{1}{(t)_\infty}.
\]
Using the $q$-binomial theorem,
\[
\sum_{n\ge 0}\frac{(a)_n}{(q)_n}t^n
=
\frac{(at)_\infty}{(t)_\infty},
\]
we then obtain
\[
L_1(zx;z^{-1}y;q,z)
=
\frac{\theta(z;q)}{(zx,qz^{-1}y)_\infty}
\frac{1}{(-y)_\infty}
\sum_{i\ge 0}
\frac{(-y)_i(-x)^i}{(q)_i}
=
\frac{\theta(z;q)(xy)_\infty}
{(zx,qz^{-1}y,-x,-y)_\infty}.
\]
This completes the proof.
\end{proof}
As an application of Corollary \ref{corrama11}, we obtain the following identity.

\begin{proposition}\label{promr1}
For all $y,z\in\bc$ with $|y|>|q|$, $z\neq 0$, and
$1/y,-z/y\not\in \{1,q,q^{2},\ldots\}$, we have
\begin{align*}
\sum_{n\ge 1}\frac{z^n q^{\binom{n}{2}}}{(y)_n}
+(q/y)_\infty\sum_{\ell\ge 0}\frac{(q/y)^\ell}{(q)_\ell(1+z^{-1}yq^{\ell})}
=\frac{\theta(-z; q)}{(y, -y/z)_\infty}.
\end{align*}
\end{proposition}

\begin{proof}
Letting $x\to 0$ in Corollary \ref{corrama11}, we obtain
\begin{align*}
\sum_{n\in\bz}\frac{z^n q^{n^2/2}}{(-q^{1/2}zy)_n}
=\frac{\theta(-q^{1/2}z; q)}{(-q^{1/2}zy, y)_\infty}.
\end{align*}
Using the reflection relation \eqref{eqrf} of $q$-shifted factorials, and then making the substitutions
$y\mapsto -q^{-1/2}y/z$ and $z\mapsto q^{-1/2}z$,
the above identity can be rewritten as
\begin{align*}
\sum_{n\ge 1}\frac{z^n q^{\binom{n}{2}}}{(y)_n}
+\sum_{n\ge 0}(q/y)_n (-y/z)^n
=\frac{\theta(-z; q)}{(y, -y/z)_\infty}.
\end{align*}
Moreover, for $|y/z|<1$ and $|y|>1$, the $q$-binomial theorem yields
\begin{align*}
\sum_{n\ge 0}(q/y)_n (-y/z)^n
&=(q/y)_\infty\sum_{n\ge 0}\frac{(-y/z)^n}{(y^{-1}q^{1+n})_\infty}\\
&=(q/y)_\infty\sum_{n\ge 0}(-y/z)^n\sum_{\ell\ge 0}\frac{(y^{-1}q^{1+n})^\ell}{(q)_\ell}\\
&=(q/y)_\infty\sum_{\ell\ge 0}\frac{(q/y)^\ell}{(q)_\ell}
\sum_{n\ge 0}(-z^{-1}yq^{\ell})^n\\
&=(q/y)_\infty\sum_{\ell\ge 0}\frac{(q/y)^\ell}{(q)_\ell(1+z^{-1}yq^{\ell})}.
\end{align*}
Thus, for all $|q^{1/2}|<|y|<|z|$, we have
\begin{align*}
\sum_{n\ge 1}\frac{z^n q^{\binom{n}{2}}}{(y)_n}
+(q/y)_\infty\sum_{\ell\ge 0}\frac{(q/y)^\ell}{(q)_\ell(1+z^{-1}yq^{\ell})}
=\frac{\theta(-z; q)}{(y, -y/z)_\infty}.
\end{align*}
Since both sides of the above identity are analytic functions for all
$y,z\in\bc$ with $|y|>|q|$, $z\neq 0$, and
$1/y,-z/y\not\in \{1,q,q^{2},\ldots\}$, the proposition follows by analytic continuation.
\end{proof}
Let $0< q<y<1$ and $z>0$. Then it is clear that
\begin{align*}
\frac{1}{1+y/z}
&=\frac{1}{1+y/z} (q/y)_\infty
\sum_{\ell\ge 0}\frac{(q/y)^\ell}{(q)_\ell}\\
&\le \sum_{\ell\ge 0}
\frac{(q/y)_\infty (q/y)^\ell}
{(q)_\ell(1+z^{-1}yq^{\ell})}
\le  (q/y)_\infty
\sum_{\ell\ge 0}\frac{(q/y)^\ell}{(q)_\ell}
=1.
\end{align*}
Moreover, for $q<u<1$ and $z>0$, there exists a constant
$\kappa_z>0$, depending only on $z$, such that
\[
\sum_{n\ge 1}\frac{z^n q^{\binom{n}{2}}}{(u)_n}
\ge
\sum_{n\ge 1}\frac{z^n q^{\binom{n}{2}}}{(q)_n}
\ge e^{\kappa_z/t},
\]
by Theorem \ref{macmain} of McIntosh \cite{MR1618298}. This yields the following asymptotic formula.

\begin{corollary}\label{cor32}
For all $z>0$ and $0<q<y<1$, there exists a constant
$\kappa_z>0$, depending only on $z$, such that
\begin{align*}
\sum_{n\ge 1}\frac{z^n q^{\binom{n}{2}}}{(y)_n}
=
\frac{\theta(-z; q)}{(y, -y/z)_\infty}
\left(1+O(e^{-\kappa_z/t})\right).
\end{align*}
\end{corollary}

We note that Corollary \ref{cor32} in fact explains the following
asymptotic formula of McIntosh \cite[p. 420]{MR1618298}:
\begin{align}\label{asymm}
\left(\sum_{n\ge 1}
\frac{(aq^{2c})^nq^{n(n-1)}}{(q;q^2)_n}\right)
&\left(\sum_{n\ge 1}
\frac{(aq^{2c})^{-n}q^{n(n-1)}}{(q;q^2)_n}\right)\nonumber\\
&=\frac{\pi}{2a^ct}
\exp\left(
\frac{{\pi^2}/{3}+\log^2 a}{4t}
+\left(c^2+\frac{2}{3}\right)t
+O(t^8)
\right).
\end{align}
In particular, replacing $z$ by $z^{-1}$ and $y$ by $q/y$ in
Corollary \ref{cor32}, we obtain
\[
\sum_{n\ge 1}\frac{z^n q^{\binom{n}{2}}}{(y)_n}
\sum_{n\ge 1}\frac{z^{-n} q^{\binom{n}{2}}}{(q/y)_n}
=
\frac{(q)_\infty^2 \theta(-z; q)\theta(-1/z;q)}
{\theta(y;q)\theta(-y/z; q)}
\left(1+O(e^{-\kappa_z/t}+e^{-\kappa_{z^{-1}}t})\right).
\]
Finally, replacing $q$ by $q^2$, $z$ by $aq^{2c}$, and $y$ by $q$,
and using the asymptotics for the theta functions, we obtain an
explanation of \eqref{asymm}. Furthermore, we see that the term
$t^8$ in the big-$O$ of \eqref{asymm} can be improved to an
exponentially small function, namely $e^{-K_a/t}$, where $K_a>0$
depends only on $a$.
\subsection{Proof of Corollary \ref{cormm}}

In this subsection, we prove Corollary \ref{cormm}. For any formal
Laurent series $f(z)=\sum_k c_k z^k$, we define
$[z^n]f(z):=c_n$ for any $n\in\bz$. Then, using Euler's identity, we obtain
\begin{align*}
H_{2}(x;y;q)
&=\sum_{i,j\in\bz}
\frac{q^{\frac{1}{4}(i+j)^2}(-x)^i(-y)^j}{(q)_i(q)_j} \\
&= [\xi^0]\sum_{k\in\bz}q^{k^2/4}\xi^{-k}
\sum_{i\in\bz}\frac{(-x\xi)^i}{(q)_i}
\sum_{j\in\bz}\frac{(-y\xi)^j}{(q)_j} \\
&= [\xi^0]\frac{1}{(-x\xi,-y\xi)_\infty}
\sum_{k\in\bz}q^{k^2/4}\xi^{-k}.
\end{align*}
Thus, for $y=-x$, since
$(-x\xi,x\xi)_\infty=(x^2\xi^2;q^2)_\infty$, Euler's identity gives
\begin{align*}
H_{2}(x;-x;q)
&= [\xi^0]\frac{1}{(x^2\xi^2;q^2)_\infty}
\sum_{k\in\bz}q^{k^2/4}\xi^{-k} \\
&= [\xi^0]\frac{1}{(x^2\xi^2;q^2)_\infty}
\sum_{k\in\bz}q^{k^2}\xi^{-2k} \\
&= [\xi^0]\sum_{n\ge 0}
\frac{x^{2n}\xi^{2n}}{(q^2;q^2)_n}
\sum_{k\in\bz}q^{k^2}\xi^{-2k} \\
&=\sum_{n\ge 0}\frac{x^{2n}q^{n^2}}{(q^2;q^2)_n}
=(-x^2q;q^2)_\infty.
\end{align*}
Similarly, for $y=q^{1/2}x$, since
$(-x\xi,-xq^{1/2}\xi)_\infty=(-x\xi;q^{1/2})_\infty$, we have
\begin{align*}
H_{2}\left(x;xq^{1/2};q\right)
&= [\xi^0]\frac{1}{(-x\xi;q^{1/2})_\infty}
\sum_{k\in\bz}q^{k^2/4}\xi^{-k} \\
&=(xq^{1/4};q^{1/2})_\infty.
\end{align*}
We summarize the above identities in the following proposition.

\begin{proposition}\label{pro22}
We have
\[
H_{2}(x;-x;q)=(-x^2q;q^2)_\infty\;\;
\text{and}\;\;
H_{2}\left(x;xq;q^2\right)=(xq^{1/2})_\infty.
\]
\end{proposition}

\begin{proof}[Proof of Corollary \ref{cormm}]
Using the identity \eqref{eqr22} in Corollary \ref{corm1} with
$y=-x$, together with the first identity in Proposition \ref{pro22},
we obtain
\begin{align*}
L_{2}(zx;-z^{-1}x;q,z^{2})
&=\frac{
\theta\left(-zq^{-1/4};q^{1/2}\right)H_{2}(x;-x;q)
+\theta\left(zq^{-1/4};q^{1/2}\right)H_{2}(-x;x;q)}
{2(zx,-qz^{-1}x)_\infty} \\
&=\frac{
\theta\left(-zq^{-1/4};q^{1/2}\right)
+\theta\left(zq^{-1/4};q^{1/2}\right)}
{2(zx,-qz^{-1}x)_\infty}
(-x^2q;q^2)_\infty \\
&=\frac{
(-x^2q;q^2)_\infty
\sum_{\substack{n\in\bz\\ 2\mid n}}
(-zq^{-1/4})^{n}q^{\frac{n(n-1)}{4}}}
{(zx,-qz^{-1}x)_\infty} \\
&= \frac{
(-x^2q;q^2)_\infty
\sum_{n\in\bz}z^{2n}q^{n(n-1)}}
{(zx,-qz^{-1}x)_\infty} \\
&=\frac{(-x^2q;q^2)_\infty\theta(-z^2;q^2)}
{(zx,-qz^{-1}x)_\infty}.
\end{align*}
Moreover, using the identity \eqref{eqr22} in Corollary \ref{corm1}
with $y=q^{1/2}x$, then replacing $q$ by $q^2$, and combining this
with the second identity in Proposition \ref{pro22}, we obtain
\begin{align*}
L_{2}(zx;z^{-1}xq;q^2,z^{2})
&=\frac{
\theta\left(-zq^{-1/2};q\right)H_{2}(x;xq;q^2)
+\theta\left(zq^{-1/2};q\right)H_{2}(-x;-xq;q^2)}
{2(zx,q^{3}z^{-1}x;q^2)_\infty} \\
&=\frac{
\theta\left(-zq^{-1/2};q\right)(xq^{1/2})_\infty
+\theta\left(zq^{-1/2};q\right)(-xq^{1/2})_\infty}
{2(zx,q^{3}z^{-1}x;q^2)_\infty}.
\end{align*}
Finally, using the definition of $L_2(x;y;q,z)$, we immediately obtain
Corollary \ref{cormm}.
\end{proof}

\section{Asymptotic Analysis}\label{sec4}
In this section, we establish the asymptotic formulas stated in
Section \ref{sec1}. The first result concerns the functions
$f_2(a,b,c;e^{-t})$ and $\hat{f}_2(a,b,c;e^{-t})$ introduced in
Subsection \ref{sec12}.

\begin{proposition}\label{prop10}
Let $c\in\br$. Suppose either that $a>1/2$ and $b>0$, or that
$1/2<a<1$ and $b=0$. Then
\[
\frac{\hat{f}_2(a,b,c;e^{-t})}{f_2(a,b,c;e^{-t})}
=
1+O(e^{-\delta/t})
\qquad \text{as } t\to 0^+,
\]
where $\delta>0$ is a constant which, in general, depends on $a$ and
$b$, but not on $c$.
\end{proposition}

\begin{proof}
Since $a>1/2$, for $t>0$ sufficiently small, with $q=e^{-t}$, we have
$2aq^c\ge a+1/2$. For $n\ge 1$, this gives
\begin{align*}
a^nq^{cn+bn^2}(-q;q)_n
&\ge \left(aq^c q^{bn}(1+q^n)\right)^n \\
&\ge \left(\left({a}/{2}+{1}/{4}\right)q^{bn}(1+q^n)\right)^n.
\end{align*}
Let $u_0:=u_0(a,b)>0$ be the solution of
\[
\left({a}/{2}+{1}/{4}\right)e^{-bu}(1+e^{-u})
=
{a}/{2}+{3}/{4}.
\]
The function
\[
\left({a}/{2}+{1}/{4}\right)e^{-bu}(1+e^{-u})
\]
is decreasing in $u>0$. Hence, if $q^n=e^{-tn}\ge e^{-u_0}$, that is,
if $0\le t\le u_0/n$, then
\[
a^nq^{bn^2+cn}(-q;q)_n
\ge
\left(
\left({a}/{2}+{1}/{4}\right)e^{-bu_0}(1+e^{-u_0})
\right)^n
=
\left({a}/{2}+{3}/{4}\right)^n.
\]
Therefore,
\begin{align*}
f_2(a,b,c;q)
&=\sum_{\ell\ge 0}a^\ell q^{b\ell^2+c\ell}(-q;q)_\ell \\
&\ge a^nq^{bn^2+cn}(-q;q)_n
\ge
\left({a}/{2}+{3}/{4}\right)^n.
\end{align*}
On the other hand,
\begin{align*}
\hat{f}_2(a,b,c;q)-f_2(a,b,c;q)
&=
\sum_{n\ge 1}\frac{q^{bn^2-cn}}{a^n(-1;q^{-1})_{n}} \\
&\le
\sum_{n\ge 1}(2a)^{-n}q^{-cn}
=
\frac{1}{2aq^c-1}
\le
\frac{1}{a-1/2},
\end{align*}
where the last inequality again uses $2aq^c\ge a+1/2$.

Now take
\[
n=\left\lfloor \frac{u_0}{2t}\right\rfloor
=
\frac{u_0}{2t}+O(1).
\]
Then
\begin{align*}
0\le
\frac{\hat{f}_2(a,b,c;q)-f_2(a,b,c;q)}
{f_2(a,b,c;q)}
&\le
\frac{1}{a-1/2}
\left(\frac{a}{2}+\frac{3}{4}\right)^{-n} \\
&=
O\left(
\exp\left(
-\frac{u_0}{2t}
\log\left(\frac{a}{2}+\frac{3}{4}\right)
\right)
\right) \\
&=
O(e^{-\delta/t}),
\end{align*}
where $\delta>0$ depends, in general, on $a$ and $b$, but not on $c$.
This proves the proposition.
\end{proof}
The second result is the proof of Corollary \ref{cor1}, which is based
on McIntosh's Theorem \ref{macmain}.

\begin{proof}[Proof of Corollary \ref{cor1}]
Let $\alpha>1$ and $z>0$. Using McIntosh's Theorem \ref{macmain}, and
setting $q=e^{-t}$ with $t\to 0^+$, we obtain the following asymptotic
formulas:
\begin{align*}
L_\alpha(q;q,z^\alpha)
&=
\sum_{n\ge 0}
\frac{z^{\alpha n}q^{\frac{\alpha}{2}n^2-\frac{\alpha}{2}n}}{(q)_n}  \\
&\sim
\frac{
w_1^{-\alpha/2}}
{
\sqrt{w_1+\alpha(1-w_1)}
}\exp\left(
\frac{
\li_2(1-w_1)+\frac{\alpha}{2}\log^2 w_1
}{t}
\right),
\\
L_{1-1/\alpha}
\left(q;q,z^{-1}q^{\frac{3}{2}-\frac{1}{2\alpha}}\right)
&=
\sum_{n\ge 0}
\frac{z^{-n}q^{\frac{1-1/\alpha}{2}n^2+n}}{(q)_n}  \\
&\sim
\frac{
w_2
}
{
\sqrt{w_2+(1-1/\alpha)(1-w_2)}
}\exp\left(
\frac{
\li_2(1-w_2)
+\frac{1-1/\alpha}{2}\log^2 w_2
}{t}
\right),
\end{align*}
where the implied constants in the asymptotic relations depend only on
$\alpha$ and $z$. Here, $w_1$ and $w_2$ denote the unique positive roots
of
\[
z^\alpha w^\alpha+w=1
\quad\text{and}\quad
w+z^{-1}w^{1-1/\alpha}=1,
\]
respectively. Therefore, we have the asymptotic tight bound:
\begin{align*}
L_{1-1/\alpha}
\left(q;q,z^{-1}q^{\frac{3}{2}-\frac{1}{2\alpha}}\right)
\asymp
e^{c_{\alpha}(z)/t}
L_\alpha(q;q,z^\alpha),
\end{align*}
where
\[
c_{\alpha}(z)
=
\li_2(1-w_2)
+\frac{1}{2}\left(1-\frac{1}{\alpha}\right)\log^2 w_2
-\li_2(1-w_1)
-\frac{\alpha}{2}\log^2 w_1.
\]

Now, letting $x=z^{-1}q$ in the asymptotic formula of Theorem
\ref{main1}, we obtain
\begin{align*}
\frac{
(2\pi z^\alpha)^{-1/2}L_\alpha(q;q,z^\alpha)
}{
\exp\left(\frac{\alpha t}{8}
+\frac{\alpha}{2t}(\log z)^2\right)
}
&=
\frac{
L_{1-1/\alpha}
\left(q;-z^{-1}q^{\frac{3}{2}-\frac{1}{2\alpha}},q\right)
}{
(\alpha t)^{1/2}(q)_\infty
} \\
&\quad
+O\left(
\frac{1
}{
(\alpha t)^{1/2}(q)_\infty
}
e^{-\frac{2\pi^2}{\alpha t}}
L_{1-1/\alpha}
\left(q;z^{-1}q^{\frac{3}{2}-\frac{1}{2\alpha}},q\right)
\right) \\
&=
\frac{
L_{1-1/\alpha}
\left(q;-z^{-1}q^{\frac{3}{2}-\frac{1}{2\alpha}},q\right)
}{
(\alpha t)^{1/2}(q)_\infty
} 
+O\left(
e^{\frac{\pi^2}{6t}
-\frac{2\pi^2}{\alpha t}
+\frac{c_\alpha(z)}{t}}
L_\alpha(q;q,z^\alpha)
\right),
\end{align*}
where the implied constants depend only on $\alpha$ and $z$. Here, we
used the well-known asymptotic formula from \eqref{eqdedk}:
\[
(q)_\infty
=
\left({2\pi}/{t}\right)^{1/2}
e^{\frac{t}{24}-\frac{\pi^2}{6t}}
\left(1+O(e^{-4\pi^2/t})\right).
\]
Equivalently,
\begin{align*}
L_\alpha(q;q,z^\alpha)
&\left(
1+
O\left(
e^{
\frac{\pi^2}{6t}
-\frac{2\pi^2}{\alpha t}
+\frac{c_\alpha(z)}{t}
+\frac{\alpha}{2t}\log^2 z
}
\right)
\right) \\
&=
\frac{
(2\pi z^\alpha)^{1/2}
e^{\frac{\alpha t}{8}
+\frac{\alpha}{2t}\log^2 z}
}{
(\alpha t)^{1/2}(q)_\infty
}
L_{1-1/\alpha}
\left(q;q,-z^{-1}q^{\frac{3}{2}-\frac{1}{2\alpha}}\right).
\end{align*}
To complete the proof, define
\begin{align*}
\delta_\alpha(z)
&=
-\frac{1}{2\pi^2}
\left(
\frac{\pi^2}{6}
-\frac{2\pi^2}{\alpha}
+c_\alpha(z)
+\frac{\alpha}{2}\log^2 z
\right) \\
&=
\frac{1}{\alpha}
-\frac{1}{2\pi^2}
\left(
\frac{\pi^2}{6}
+c_\alpha(z)
+\frac{\alpha}{2}\log^2 z
\right).
\end{align*}
Then
\begin{align*}
L_\alpha(q;q,z^\alpha)
&=
\frac{
(2\pi z^\alpha)^{1/2}
e^{\frac{\alpha t}{8}
+\frac{\alpha}{2t}\log^2 z}
}{
(\alpha t)^{1/2}(q)_\infty
}
L_{1-1/\alpha}
\left(q;q,-z^{-1}q^{\frac{3}{2}-\frac{1}{2\alpha}}\right) \\
&\quad \times
\left(
1+O\left(e^{-2\pi^2\delta_\alpha(z)/t}\right)
\right).
\end{align*}

It remains to simplify the expression for $\delta_\alpha(z)$. Write
\[
w_\alpha=(zw_1)^\alpha
\quad\text{and}\quad
w_\beta=w_2.
\]
Then, with $\beta=(1-1/\alpha)^{-1}$, the quantities $w_\alpha$ and
$w_\beta$ solve the equations
\[
w+z^{-1}w^{1/\alpha}=1
\quad\text{and}\quad
w+z^{-1}w^{1/\beta}=1,
\]
respectively. Hence, using the identity
\[
\li_2(u)+\li_2(1-u)+(\log u)\log(1-u)=\frac{\pi^2}{6}
\]
from \cite[Equation (1.11)]{MR618278}, we obtain
\begin{align*}
c_{\alpha}(z)
&=
\li_2(1-w_\beta)
+\frac{1}{2\beta}\log^2 w_\beta
-\li_2(w_\alpha)
-\frac{\alpha}{2}
\log^2\left(z^{-1}w_\alpha^{1/\alpha}\right) \\
&=
\frac{\pi^2}{6}
-(\log w_\beta)\log(1-w_\beta)
-\li_2(w_\beta)
+\frac{1}{2\beta}\log^2 w_\beta
-\li_2(w_\alpha)
-\frac{\alpha}{2}
\log^2\left(z^{-1}w_\alpha^{1/\alpha}\right) \\
&=
\frac{\pi^2}{6}
-\li_2(w_\alpha)
-\li_2(w_\beta)
-(\log w_\beta)\log\left(z^{-1}w_\beta^{1/\beta}\right)
+\frac{1}{2\beta}\log^2 w_\beta  \\
&\quad
-\frac{\alpha}{2}
\log^2\left(z^{-1}w_\alpha^{1/\alpha}\right) \\
&=
\frac{\pi^2}{6}
-\li_2(w_\alpha)
-\li_2(w_\beta)
-\frac{1}{2\beta}\log^2 w_\beta
-\frac{1}{2\alpha}\log^2 w_\alpha
+(\log z)\log(w_\alpha w_\beta)
-\frac{\alpha}{2}\log^2 z.
\end{align*}
Thus,
\[
\delta_\alpha(z)
=
\frac{1}{\alpha}
-\frac{1}{6}
+\frac{1}{2\pi^2}
\left(
\li_2(w_\alpha)
+\li_2(w_\beta)
+\frac{1}{2\beta}\log^2 w_\beta
+\frac{1}{2\alpha}\log^2 w_\alpha
-(\log z)\log(w_\alpha w_\beta)
\right),
\]
which completes the proof of the corollary.
\end{proof}

%\bibliographystyle{plain}
%\bibliography{test}

\end{document}